\documentclass[12pt,leqno,draft]{article} 
\usepackage{amsmath,amssymb,amsthm,amsfonts,bm}
\usepackage{enumerate}
\usepackage{indentfirst}
\usepackage{cite}
\usepackage[usenames,dvipsnames]{color}
\definecolor{viola}{rgb}{0.3,0,0.7}
\definecolor{ciclamino}{rgb}{0.5,0,0.5}

\def\pier #1{{\color{Green}#1}}
\def\pier #1{#1}

\makeatletter
\def\@cite#1#2{[{{\bfseries #1}\if@tempswa , #2\fi}]}
\renewcommand{\section}{%
\@startsection{section}{1}{\z@}
{0.5truecm plus -1ex minus -.2ex}%
{1.0ex plus .2ex}{\bfseries\large}}
\def\@seccntformat#1{\csname the#1\endcsname.\ }
\makeatother

\numberwithin{equation}{section} 
\theoremstyle{theorem}
\newtheorem{thm}{Theorem}[section]

\newtheorem{lem}[thm]{Lemma}

\theoremstyle{definition}
\newtheorem{df}{Definition}[section]
\newtheorem{remark}{Remark}[section]
\newtheorem*{prth2.1}{Proof of Theorem 2.1}
\newtheorem*{prth2.2}{Proof of Theorem 2.2}
\newtheorem*{prth2.3}{Proof of Theorem 2.3}

\newcommand{\ep}{\varepsilon}

\newcommand{\RN}{\mathbb{R}^N}

\def\dt{\partial_t}

\def\checkmmode #1{\relax\ifmmode\hbox{#1}\else{#1}\fi}

\let\hat\widehat
\def\Beta{\hat{\vphantom t\smash\beta\mskip2mu}\mskip-1mu}
\def\Pi{\hat\pi}

\def\thetaz{\theta_0}
\def\phiz{\phi_0}

\def\erre{{\mathbb{R}}}

\let\non\nonumber

\def\aand{\quad\hbox{and}\quad}

\begin{document}
\begin{center}
    \Large{{\bf Time discretization of a nonlinear phase field system \\
in general domains
           }}
\end{center}
\vspace{5pt}
\begin{center}
    Pierluigi Colli\\
    \vspace{2pt}
    Dipartimento di Matematica ``F. Casorati", Universit\`a di Pavia \\ 
    and Research Associate at the IMATI -- C.N.R. Pavia \\ 
    via Ferrata 5, 27100 Pavia, Italy\\
    {\tt pierluigi.colli@unipv.it}\\
    \vspace{12pt}
    Shunsuke Kurima%
   \footnote{Corresponding author}\\
    \vspace{2pt}
    Department of Mathematics, 
    Tokyo University of Science\\
    1-3, Kagurazaka, Shinjuku-ku, Tokyo 162-8601, Japan\\
    {\tt shunsuke.kurima@gmail.com}\\
    \vspace{2pt}
\end{center}
\begin{center}    
    \small \today
\end{center}

\vspace{2pt}
\newenvironment{summary}
{\vspace{.5\baselineskip}\begin{list}{}{%
     \setlength{\baselineskip}{0.85\baselineskip}
     \setlength{\topsep}{0pt}
     \setlength{\leftmargin}{12mm}
     \setlength{\rightmargin}{12mm}
     \setlength{\listparindent}{0mm}
     \setlength{\itemindent}{\listparindent}
     \setlength{\parsep}{0pt}
     \item\relax}}{\end{list}\vspace{.5\baselineskip}}
\begin{summary}
{\footnotesize {\bf Abstract.}
    This paper deals with 
    the nonlinear phase field system 
    \begin{equation*}
     \begin{cases}
         \partial_t (\theta +\ell \varphi) - \Delta\theta = f  
         & \mbox{in}\ \Omega\times(0, T),
 \\[1mm]
         \partial_t \varphi - \Delta\varphi + \xi + \pi(\varphi) 
         = \ell \theta,\ \xi\in\beta(\varphi)   
         & \mbox{in}\ \Omega\times(0, T) 
     \end{cases}
 \end{equation*}
    %
    in a {\it general} domain $\Omega\subseteq\RN$. 
    Here $N \in \mathbb{N}$, $T>0$, $\ell>0$, 
    $f$ is a source term, 
    $\beta$ is a maximal monotone graph and 
    $\pi$ is a Lipschitz continuous function. 
    We note that in the above system the nonlinearity $\beta+\pi$ 
    replaces the derivative of a potential of double well type. 
    Thus it turns out that 
    the system is a generalization of the Caginalp phase field model 
    and it has been studied by many authors 
    in the case that $\Omega$ is a bounded domain.   
    However, for unbounded domains  
    the analysis of the system seems to be at an early stage. 
    In this paper we study the existence of solutions  
    by employing a time discretization scheme 
    and passing to the limit as the time step $h$ goes to $0$. 
    In the limit procedure we face with the difficulty that 
    the embedding $H^1(\Omega) \hookrightarrow L^2(\Omega)$ 
    is not compact in the case of unbounded domains. 
    Moreover, we can prove an interesting error estimate of order $h^{1/2}$ 
    for the difference between continuous and discrete solutions. 
    }
\end{summary}
\vspace{10pt}

\newpage

\section{Introduction and results} \label{Sec1}

In the present contribution we address a nonlinear phase field system 
in a {\it general} domain $\Omega \subseteq \RN$ 
and discuss it using a time discretization procedure, 
which turns out to be useful and efficient 
for the approximation and the existence proof. 
Moreover, we show an error estimate of order $1/2$ 
in suitable norms for the difference between continuous and discrete solutions. 

Our analysis moves from the consideration of 
the following simple version of 
the phase-field system of Caginalp type 
(cf.\ \cite{Cag, EllZheng}; one may also see 
the monographs \cite{BrokSpr, fremond, V1996}): 
\begin{align}
  & \dt \bigl( c_0 \theta + \ell \phi \bigr) - \kappa \Delta\theta = f
  & \quad \hbox{in $\Omega\times(0, T)$}, 
  \label{Iprima}
  \\
  & \dt\phi - \eta \Delta\phi + F'(\phi) = \ell \theta 
  & \quad \hbox{in $\Omega\times(0, T)$}, 
  \label{Iseconda}
\end{align}
where $\Omega$ is the three-dimensional domain in which the evolution takes place,
$T$ is some final time,
$\theta$ refers to the relative temperature around some critical value, 
that is taken to be $0$ without loss of generality,
and $\phi$ is the order parameter.
Moreover, $c_0$, $\ell$, $\kappa$ and $\eta$ are positive constants,
$f$ is a source term
and $F'$ is the derivative of a smooth double-well potential $F$, 
whose prototype reads 
\begin{align}\label{regpot}
  & F_{reg}(r) = \frac 14 \, (r^2-1)^2 \,, \quad r \in \erre. 
\end{align}
In our approach, we aim to consider other kinds of potentials, 
which are by now widely known and extensively used 
in the mathematical literature, that are the logarithmic potential and 
the double obstacle potential: 
\begin{align}
  & F_{log}(r) = \bigl( (1+r)\ln (1+r)+(1-r)\ln (1-r) \bigr) - c_1 \, r^2 
  \quad \hbox{if $r \in (-1,1)$}, 
  \label{logpot}
  \\
  & F_{obs}(r) = I(r) - c_2 \, r^2 \,,
  \quad r \in \erre. 
  \label{obspot}
\end{align}
Here, the coefficient $c_1$ in \eqref{logpot}
is larger than $1$, so that $F_{log}$ admits a double well, 
and 
$c_2$ in \eqref{obspot} is an arbitrary positive constant,  
whereas the function $I$ in \eqref{obspot} is the indicator function of $[-1,1]$, i.e.,
it takes the values $0$ or $+\infty$ according to whether or not $r$ belongs to 
$[-1,1]$.
Note that $F_{log}$ can be extended by continuity to the closed interval $[-1, 1]$, 
but its derivative $F_{log}'$ turns out to be singular 
as the variable approaches $-1$ from the right 
or $+1$ from the left. 
On the other hand, $F_{obs}$ is even a non-smooth potential and for it 
it is no longer possible to consider the derivative, 
but we have to deal with the subdifferential of $I$. 
Hence, 
in order to be as general as possible, 
in our investigation we allow the potential $F$ to be just the sum 
$$F=\Beta+\Pi,$$
where $\Beta$ is a convex function that is allowed to take the value $+\infty$ 
somewhere, 
and $\Pi$ is a smooth perturbation which may be concave.   
In such a case, $\Beta$ is supposed to be proper and lower semicontinuous, 
so that its subdifferential is well defined and can replace the derivative
which might not exist. 
Let us point out that, in the case of a multivalued subdifferential like $\partial I$, 
equation \eqref{Iseconda} becomes a differential inclusion. 

The system \eqref{Iprima}-\eqref{Iseconda} is of course complemented 
by initial conditions like $\theta(0)=\thetaz$ and $\phi(0)=\phiz$,   
and suitable boundary conditions.
Concerning the latter, 
we set the usual homogeneous Neumann condition 
for both $\theta$ and $\phi$, that is,
\begin{equation}
  \partial_{\nu} \theta = 0
  \aand
  \partial_{\nu} \phi = 0
  \quad \hbox{on $\partial\Omega\times(0, T)$}, 
  \non
\end{equation}
where $\partial_{\nu}$ denotes 
differentiation with respect to the outward normal of $\partial\Omega$. 
Indeed, by these conditions we are assuming that 
there is no flow exchange with the exterior of $\Omega$. 

Equations \eqref{Iprima}-\eqref{Iseconda} yield a system of phase field type. 
Such systems have been introduced (cf. \cite{BrokSpr, Cag, V1996}) 
in order to include phase dissipation effects 
in the dynamics of moving interfaces arising in thermally induced phase transitions 
(the reader may also see \cite{BCJV2012, CCE2011, M2014, MQ2011, MQ2009}). 
In our framework, we are actually considering 
the following form for the total free energy:  
\begin{equation}
G (\theta, \varphi) = 
\int_\Omega \left( - \frac{c_0}2 \theta^2 - \ell\theta \varphi + F (\varphi) 
                                                         + \frac \eta 2 |\nabla \varphi |^2 \right), 
\label{free}
\end{equation}
where $c_0 $ and $\ell$ stand for the specific heat and latent heat coefficients,
respectively, with a terminology motivated by earlier studies (see \cite{duvaut})
on the Stefan problem; let us also refer to the monography \cite{fremond},  
which deals with phase change models as well.  In this connection, 
it is worth to introduce the enthalpy $e$ by
\begin{equation}\label{aboute}
e= - \frac{\delta G}{\delta \theta}, 
\end{equation}
where $\frac{\delta G}{\delta \theta}$ denotes 
the variational derivative of $G $ with respect to $\theta$, 
so that \eqref{aboute} yields $e=c_0 \theta +\ell \varphi$. 
Then the governing balance and phase equations are given by 
\begin{gather}
\dt e + \mbox{div}\, {\bf q}  = f,  
\label{1phys}
\\
\dt \varphi +  \frac{\delta G}{\delta\varphi} =0, 
\label{2phys}
\end{gather}
where ${\bf q} $ denotes the thermal flux vector, 
$f$ represents some heat source and 
$\frac{\delta G}{\delta\varphi}$ is 
the variational derivative of $G$ with respect to $\varphi$.   
Hence \eqref{2phys} reduces 
exactly to \eqref{Iseconda} along with 
the homogeneous Neumann boundary condition for $\varphi$. 
Moreover, if we assume the classical Fourier law 
$ {\bf q} = - \kappa \, \nabla \theta $,
then  \eqref{1phys} is nothing but the usual energy balance equation of the 
Caginalp model.  
Moreover, \eqref{Iprima} follows from  \eqref{1phys} and 
the Neumann boundary condition for $\theta$ is a consequence of 
the no-flux condition $ {\bf q} \cdot {\bf n} =0 $ on the boundary. 
We \pier{notice} that the above phase field system has received 
a good deal of attention in the last decades 
\cite{ABHR2013, CCE2008, CMM2015, CM2009, CGMQ2017, MM2016} 
and 
can be deduced as a special gradient-flow problem 
(cf., e.g., \cite{RS06} and references therein).

\pier{Let us} point out that questions related to the well-posedness, 
long-time behaviour of solutions and optimal control problems have been 
investigated for the Caginalp system \eqref{Iprima}-\eqref{Iseconda} 
and for some variation or extension of this phase field system. 
Without any sake of completeness, \pier{we} mention the contributions 
\cite{B2013, EllZheng, CC2013, CGM2012, CC2012, CHIS2009, 
CGM2013, GraPetSch, KenmNiez, MN2018, MQ2013, M2015} 
for various qualitative analyses  
and \cite{BCGMR, CGM, CoGiMaRo, HoffJiang, HKKY}
for some related control problems. 

For the sake of simplicity, 
in the sequel of the paper 
we simply let $c_0 = \kappa = \eta = 1$ by observing that 
our treatment can be easily extended to the case of coefficients different from $1$. 
On the other hand, 
we keep the parameter $\ell$ in both equations \eqref{Iprima} and \eqref{Iseconda} 
since $\ell$ plays a role in the estimates. 

The case of unbounded domains 
has the difficult mathematical point that 
compactness methods cannot be applied directly 
(related discussions can be found, e.g., in \cite{FKY-2017, K2, K3, KY1, KY2}).  
It would be interesting to construct an applicable theory for the case of 
unbounded domains and to set assumptions for the case of unbounded domains 
by trying to keep some typical features  
in previous works, that is, in the case of bounded domains. 
By considering the case of unbounded domains, 
it may be possible to make a new finding 
which is not given or known in the case of bounded domains.  
Also, the new finding would be useful 
for other studies of partial differential equations. 

In this paper we consider 
the initial-boundary value problem on a {\it general} domain  
for the nonlinear phase field system 
%
%
 \begin{equation*}\tag*{(P)}\label{P}
     \begin{cases}
         \partial_t (\theta +\ell \varphi) - \Delta\theta = f   
         & \mbox{in}\ \Omega\times(0, T), 
 \\[2mm]
         \partial_t \varphi - \Delta\varphi + \xi + \pi(\varphi) 
         = \ell \theta,\ \xi \in \beta(\varphi)   
         & \mbox{in}\ \Omega\times(0, T), 
 \\[2mm]
         \partial_{\nu}\theta= \partial_{\nu}\varphi = 0                                   
         & \mbox{on}\ \partial\Omega\times(0, T),
 \\[1mm]
        \theta(0) = \theta_0,\ \varphi(0)=\varphi_0                                          
         & \mbox{in}\ \Omega, 
     \end{cases}
 \end{equation*}
where $\Omega$ is a {\it bounded} domain or an {\it unbounded} domain in $\RN$ 
with smooth bounded boundary $\partial\Omega$ 
(e.g., $\Omega = \mathbb{R}^{N}\setminus \overline{B(0, R)}$,  
where $B(0, R)$ is the open ball with center $0$ and radius $R>0$) 
or $\Omega=\mathbb{R}^{N}$ 
(in this case, no boundary condition should be prescribed) 
or $\Omega=\mathbb{R}_{+}^{N}$.  
Moreover, we let $\ell>0$  
and deal with the following conditions (A1)-(A4): 
%
%
%
\begin{enumerate} 
\setlength{\itemsep}{0mm}
\item[(A1)] $\beta \subset \mathbb{R}\times\mathbb{R}$                                
is a maximal monotone graph with effective domain $D(\beta)$ 
and $\beta(r) = \partial\widehat{\beta}(r)$, where 
$\partial\widehat{\beta}$ denotes the subdifferential of 
a proper lower semicontinuous convex function 
$\widehat{\beta} : \mathbb{R} \to [0, +\infty]$ satisfying $\widehat{\beta}(0) = 0$.    
\item[(A2)] $\pi : \mathbb{R} \to \mathbb{R}$ is                         
a Lipschitz continuous function and $\pi(0) = 0$. 
Moreover, there exists a function $\widehat{\pi} \in C^1(\mathbb{R})$ 
such that $\pi = \widehat{\pi}'$.  
\item[(A3)] $f \in L^2(0, T; L^2(\Omega))$.  
\item[(A4)] $\theta_0, \varphi_0 \in H^1(\Omega)$  
and $\widehat{\beta}(\varphi_{0}) \in L^1(\Omega)$.    
\end{enumerate} 

Please note that 
the three functions in \eqref{regpot}-\eqref{obspot} actually satisfy (A1) and (A2), 
indeed we have that  
\begin{align*}
&\widehat{\beta}(r)=\frac{1}{4}r^4,\ \beta(r)=r^3,\ 
\widehat{\pi}(r)=\frac{1}{4}(-2r^2+1),\ \pi(r)=-r,\ r \in \mathbb{R}, 
\quad\mbox{in \eqref{regpot};} 
\\[7mm] \notag 
&\widehat{\beta}(r)=
\begin{cases}
(1+r)\log(1+r) + (1-r)\log(1-r)  &\mbox{if}\ r \in (-1, 1), 
\\ 
2\log 2  &\mbox{if}\ r \in \{-1, 1\},  
\\ 
+\infty &\mbox{otherwise}, 
\end{cases}
\\
&\beta(r)=\log\frac{1+r}{1-r}\quad \mbox{provided that }\  r \in (-1, 1), 
\\ 
&\widehat{\pi}(r)=-c_1 r^2,\ \pi(r)=-2c_1 r,\ r \in \mathbb{R}, 
\quad\mbox{in \eqref{logpot};} 
\\[7mm] \notag  
&\widehat{\beta}(r)=I(r),\ r \in \mathbb{R},\  
\beta(r)=
\begin{cases}
0  &\mbox{if}\ r \in (-1, 1), 
\\ 
[0, \infty)  &\mbox{if}\ r=1, 
\\ 
(-\infty, 0] &\mbox{if}\ r=-1, 
\end{cases}
\\ 
&\widehat{\pi}(r)=-c_2 r^2,\ \pi(r)=-2c_2 r,\ r \in \mathbb{R}, 
\quad\mbox{in \eqref{obspot}.} 
\end{align*}

\smallskip

\pier{Let us define the Hilbert spaces 
   $$
   H:=L^2(\Omega), \quad V:=H^1(\Omega)
   $$
 with inner products 
 \begin{align*} 
 &(u_{1}, u_{2})_{H}:=\int_{\Omega}u_{1}u_{2}\,dx \quad  (u_{1}, u_{2} \in H), \\
 &(v_{1}, v_{2})_{V}:=
 \int_{\Omega}\nabla v_{1}\cdot\nabla v_{2}\,dx + \int_{\Omega} v_{1}v_{2}\,dx \quad 
 (v_{1}, v_{2} \in V),
\end{align*}
 respectively,
 and with the related Hilbertian norms.}
 Moreover, \pier{we use the notation}
   $$
   W:=\bigl\{z\in H^2(\Omega)\ |\ \partial_{\nu}z = 0 \quad 
   \mbox{a.e.\ on}\ \partial\Omega\bigr\}.
   $$ 

This paper is organized as follows. 
In Section \ref{Sec2} we introduce a time discretization of \ref{P}, 
set precisely the approximate problem,  
and state the main theorems. 
Section \ref{Sec3} contains 
the proof of the existence for the discrete problem. 
In Section \ref{Sec4} we establish uniform estimates for the approximate problem 
and pass to the limit. 
Section \ref{Sec5} show error estimates between solutions of \ref{P} 
and solutions of the approximate problem. 


\section{Time discretization and main results}\label{Sec2}

We will prove existence of solutions to \ref{P} 
by employing a time discretization scheme.   
More precisely, we will establish existence for \ref{P} 
by passing to the limit in the problem 
%
%
%
 \begin{equation*}\tag*{(P)$_{n}$}\label{Pn}
     \begin{cases}
         \delta_h \theta_n + \ell\delta_h \varphi_n-\Delta\theta_{n+1} = f_{n+1},  
     \\[2mm] 
         \delta_h \varphi_n -\Delta\varphi_{n+1} + \xi_{n+1} + \pi(\varphi_{n+1}) 
         = \ell\theta_n,\ \xi_{n+1} \in \beta(\varphi_{n+1}) 
     \end{cases}
 \end{equation*}
for $n=0, ..., N-1$ 
as $h\searrow0$, 
where $h=\frac{T}{N}$, $N \in \mathbb{N}$, 
\begin{align}\label{deltah}
\delta_h \theta_n := \dfrac{\theta_{n+1}-\theta_n}{h},\ 
\delta_h \varphi_n := \dfrac{\varphi_{n+1}-\varphi_n}{h}, 
\end{align}
and $f_k:=\frac{1}{h}\int_{(k-1)h}^{kh} f(s)\,ds$ for $k=1, ..., N$. 
Also, putting 
\begin{align}
& \widehat{\theta}_h (0) := \theta_0,\ 
   \partial_t \widehat{\theta}_h (t) := \delta_h \theta_n, \quad 
   \widehat{\varphi}_h (0) := \varphi_0,\ 
   \partial_t \widehat{\varphi}_h (t) := \delta_h \varphi_n, \label{hat} 
\\ 
&\overline{\theta}_h (t) := \theta_{n+1},\ \overline{f}_h (t) := f_{n+1},\ 
  \overline{\varphi}_h (t) := \varphi_{n+1},\ 
  \overline{\xi}_h (t) := \xi_{n+1},\ \underline{\theta}_h (t) := \theta_{n} 
\label{overandunderline}   
\end{align}
for a.a.\ $t \in (nh, (n+1)h)$, $n=0, ..., N-1$, 
we can rewrite \ref{Pn} as the problem 
%
%
%
 \begin{equation*}\tag*{(P)$_{h}$}\label{Ph}
     \begin{cases}
         \partial_t \widehat{\theta}_h + \ell\partial_t \widehat{\varphi}_h 
         -\Delta\overline{\theta}_h = \overline{f}_h 
         & \mbox{in}\ \Omega\times(0, T),
     \\[0mm] 
        \partial_t \widehat{\varphi}_h - \Delta\overline{\varphi}_h 
        + \overline{\xi}_h + \pi(\overline{\varphi}_h) = \ell\underline{\theta}_h,\ 
        \overline{\xi}_h \in \beta(\overline{\varphi}_h) 
         & \mbox{in}\ \Omega\times(0, T), 
     \\ 
         \widehat{\theta}(0) = \theta_0,\ \widehat{\varphi}(0)=\varphi_0                                          
         & \mbox{in}\ \Omega.  
     \end{cases}
 \end{equation*}

\begin{remark}
We have 
\begin{align}
&\|\widehat{\theta}_h\|_{L^2(0, T; H)}^2 
\leq h\|\theta_0\|_{H}^2 + 2\|\overline{\theta}_h\|_{L^2(0, T; H)}^2, \label{rem1} 
\\ 
&\|\widehat{\varphi}_h\|_{L^2(0, T; H)}^2 
\leq h\|\varphi_0\|_{H}^2 + 2\|\overline{\varphi}_h\|_{L^2(0, T; H)}^2, \label{rem2}
\\[1mm] 
&\|\widehat{\theta}_h\|_{L^{\infty}(0, T; V)} 
= \max\{\|\theta_{0}\|_{V}, \|\overline{\theta}_h\|_{L^{\infty}(0, T; V)}\}, \label{rem3} 
\\[1mm] 
&\|\widehat{\varphi}_h\|_{L^{\infty}(0, T; V)} 
= \max\{\|\varphi_{0}\|_{V}, \|\overline{\varphi}_h\|_{L^{\infty}(0, T; V)}\}, 
\label{rem4}
\\ 
&\|\overline{\theta}_h - \widehat{\theta}_h\|_{L^2(0, T; H)}^2 
= \frac{h^2}{3}\|\partial_t \widehat{\theta}_h\|_{L^2(0, T; H)}^2, \label{rem5}
\\ 
&\|\overline{\varphi}_h - \widehat{\varphi}_h\|_{L^2(0, T; H)}^2 
= \frac{h^2}{3}\|\partial_t \widehat{\varphi}_h\|_{L^2(0, T; H)}^2, \label{rem6}
\end{align}
as the reader can check directly 
by using the definitions \eqref{hat} and \eqref{overandunderline}.   
\end{remark}
\smallskip

We define solutions of \ref{P} as follows. 

\begin{df}
A triplet $(\theta, \varphi, \xi)$ with 
\begin{align*}
&\theta \in H^1(0, T; H) \cap L^{\infty}(0, T; V) \cap L^2(0, T; W), \\
&\varphi \in H^1(0, T; H) \cap L^{\infty}(0, T; V) \cap L^2(0, T; W), \\
&\xi \in L^2(0, T; H) 
\end{align*}
is called a solution of \ref{P} if $(\theta, \varphi, \xi)$ satisfies 
\begin{align}
&\partial_t (\theta +\ell \varphi) - \Delta\theta = f   
       \quad \mbox{a.e.\ on}\ \Omega\times(0, T), \label{df1} 
\\ 
&\partial_t \varphi - \Delta\varphi + \xi + \pi(\varphi) = \ell \theta,\ 
                                                                      \xi \in \beta(\varphi)   
         \quad \mbox{a.e.\ on}\ \Omega\times(0, T), \label{df2} 
\\ 
&\theta(0) = \theta_0,\ \varphi(0)=\varphi_0                                          
        \quad \mbox{a.e.\ on}\ \Omega. \label{df3}
\end{align}
\end{df}
\smallskip

Now the main results read as follows.

\begin{thm}\label{maintheorem1}
Assume that {\rm (A1)-(A4)} hold. 
Then 
for all 
$h \in \left(0, \frac{1}{\|\pi'\|_{L^{\infty}(\mathbb{R})}}\right)$ 
there exists a unique solution $(\theta_{n+1}, \varphi_{n+1}, \xi_{n+1})$ 
of {\rm \ref{Pn}} 
satisfying 
$$
\theta_{n+1}, \varphi_{n+1} \in W\quad 
and\quad \xi_{n+1} \in H \qquad\mbox{for}\ n=0, ..., N-1.
$$ 
\end{thm}

\begin{thm}\label{maintheorem2}
Assume that {\rm (A1)-(A4)} hold. 
Then there exists a unique solution of {\rm \ref{P}}. 
\end{thm}

\begin{thm}\label{erroresti} 
Assume that {\rm (A1)-(A4)} hold. 
Assume further that 
$f \in W^{1,1}(0, T; H)$. 
Then 
there exists 
$h_0 \in \left(0, \min\left\{1, \frac{1}{\|\pi'\|_{L^{\infty}(\mathbb{R})}}\right\}\right)$ 
such that  
for all $\ell >0$ 
there exists a constant $M_1=M_1(\ell, T)>0$ such that 
\begin{align}\label{erroresti1}
&\ell\|\widehat{\varphi}_h - \varphi \|_{L^{\infty}(0, T; H)} 
+ \ell\|\overline{\varphi}_h - \varphi\|_{L^2(0, T; V)} \\ \notag
&+ \|\widehat{\theta}_h - \theta 
                                     + \ell(\widehat{\varphi}_h - \varphi)\|_{L^{\infty}(0, T; H)} 
+ \|\overline{\theta}_h - \theta\|_{L^2(0, T; V)}  \\ \notag 
&\leq M_1 h^{1/2} 
\end{align}
for all $h \in (0, h_0)$.  
In particular, 
for all $\ell >0$ 
there exists a constant $M_2=M_2(\ell, T)>0$ such that 
\begin{align*}
 \|\widehat{\theta}_h - \theta \|_{L^{\infty}(0, T; H)} 
 \leq M_2 h^{1/2} 
\end{align*} 
for all $h \in (0, h_0)$.   
\end{thm} 

\vspace{10pt}

\section{Existence of discrete solution}\label{Sec3}

In this section we will prove Theorem \ref{maintheorem1}.

\begin{lem}\label{about a elliptic eq} 
For all $g\in H$ and all  
$h \in \left(0, \frac{1}{\|\pi'\|_{L^{\infty}(\mathbb{R})}}\right)$ 
there exists a unique solution $(\varphi, \xi)$ of the equation 
$\varphi - h\Delta\varphi + \xi + h\pi(\varphi) = g$, $\xi \in \beta(\varphi)$,  
satisfying $\varphi \in W$ and $\xi \in H$.  
\end{lem}
\begin{proof}
Let $\ep>0$ and 
let $\beta_{\ep}$ be the Yosida approximation of $\beta$ on $\mathbb{R}$.  
Then the operator $1-h\Delta+h\beta_{\ep} + h\pi : V \to V^{*}$ 
is monotone, continuous and coercive 
for all $h \in \left(0, \frac{1}{\|\pi'\|_{L^{\infty}(\mathbb{R})}}\right)$. 
Indeed, we have 
\begin{align*}
\langle \psi - h\Delta\psi + h\beta_{\ep}(\psi) + h\pi(\psi), \psi \rangle_{V^{*}, V} 
&\geq (1-h\|\pi'\|_{L^{\infty}(R)})\|\psi\|_{H}^2 + h\|\nabla\psi\|_{H}^2 \\ \notag
&\geq \min\{1-h\|\pi'\|_{L^{\infty}(R)}, h\}\|\psi\|_{V}^2
\end{align*}
for all $\psi \in V$. 
Thus the operator $1-h\Delta+h\beta_{\ep} + h\pi : V \to V^{*}$ 
is surjective 
for all $h \in \left(0, \frac{1}{\|\pi'\|_{L^{\infty}(\mathbb{R})}}\right)$ 
(see e.g., \cite[p.\ 37]{Barbu}) 
and then the elliptic regularity theory yields that 
for all $g \in H$ 
and all $h \in \left(0, \frac{1}{\|\pi'\|_{L^{\infty}(\mathbb{R})}}\right)$
there exists a unique solution $\varphi_{\ep} \in W$ of the equation 
\begin{align}\label{ep elliptic eq}
\varphi_{\ep} - h\Delta\varphi_{\ep} 
+ h(\beta_{\ep}(\varphi_{\ep}) + \pi(\varphi_{\ep}))  = g.
\end{align}
Here, multiplying \eqref{ep elliptic eq} by $\varphi_{\ep}$ 
and integrating over $\Omega$, 
we obtain the inequality    
\begin{align*}
&\|\varphi_{\ep}\|_{H}^2 + h\|\nabla\varphi_{\ep}\|_{H}^2 
+ h(\beta_{\ep}(\varphi_{\ep}), \varphi_{\ep})_{H} 
\\ 
&= (g, \varphi_{\ep})_{H} - h(\pi(\varphi_{\ep}), \varphi_{\ep})_{H} \\ 
&\leq \frac{1}{2(1-h\|\pi'\|_{L^{\infty}(R)})}\|g\|_{H}^2 
    + \frac{1+h\|\pi'\|_{L^{\infty}(R)}}{2}\|\varphi_{\ep}\|_{H}^2 
\end{align*} 
for all $h \in \left(0, \frac{1}{\|\pi'\|_{L^{\infty}(\mathbb{R})}}\right)$, 
and hence 
for all $h \in \left(0, \frac{1}{\|\pi'\|_{L^{\infty}(\mathbb{R})}}\right)$ 
there exists a constant $C_1=C_1(h)>0$ such that 
\begin{align}\label{Vesti}
\|\varphi_{\ep}\|_{V} \leq C_1 
\end{align}
for all $\ep>0$.
We see from \eqref{ep elliptic eq} and \eqref{Vesti} that 
\begin{align*}
&\|\beta_{\ep}(\varphi_{\ep})\|_{H}^2 \\ 
&= (\beta_{\ep}(\varphi_{\ep}), \beta_{\ep}(\varphi_{\ep}))_{H} \\ 
&= -\frac{1}{h}(\beta_{\ep}(\varphi_{\ep}), \varphi_{\ep})_{H} 
   - \int_{\Omega} \beta_{\ep}'(\varphi_{\ep})|\nabla\varphi_{\ep}|^2 
   - (\pi(\varphi_{\ep}), \beta_{\ep}(\varphi_{\ep}))_{H} 
   +\frac{1}{h}(g, \beta_{\ep}(\varphi_{\ep}))_{H} \\
&\leq C_1^2\|\pi'\|_{L^{\infty}(\mathbb{R})}^2  
        + \frac{1}{h^2}\|g\|_{H}^2 + \frac{1}{2}\|\beta_{\ep}(\varphi_{\ep})\|_{H}^2  
\end{align*} 
for all $h \in \left(0, \frac{1}{\|\pi'\|_{L^{\infty}(\mathbb{R})}}\right)$. 
Hence 
for all $h \in \left(0, \frac{1}{\|\pi'\|_{L^{\infty}(\mathbb{R})}}\right)$ 
there exists a constant $C_{2}=C_{2}(h)>0$ such that 
\begin{align}\label{betaesti}
\|\beta_{\ep}(\varphi_{\ep})\|_{H} \leq C_2 
\end{align}
for all $\ep>0$. 
Moreover, \eqref{ep elliptic eq}-\eqref{betaesti} yield that 
\begin{align*}
\|\Delta \varphi_{\ep}\|_{H} 
&= \frac{1}{h}
          \|\varphi_{\ep} + h(\beta_{\ep}(\varphi_{\ep})+\pi(\varphi_{\ep}))-g\|_{H}  \\
&\leq \frac{C_1}{h} + C_2 + C_{1}\|\pi'\|_{L^{\infty}(\mathbb{R})} + \frac{1}{h}\|g\|_{H}
\end{align*} 
for all $h \in \left(0, \frac{1}{\|\pi'\|_{L^{\infty}(\mathbb{R})}}\right)$. 
Thus, 
by the elliptic regularity theory and \eqref{Vesti},   
for all $h \in \left(0, \frac{1}{\|\pi'\|_{L^{\infty}(\mathbb{R})}}\right)$
there exists a constant $C_{3}=C_{3}(h)>0$ such that 
\begin{align}\label{Lapesti}
\|\varphi_{\ep}\|_{W} \leq C_3 
\end{align} 
for all $\ep>0$. 
Therefore we infer from \eqref{betaesti} and \eqref{Lapesti} that 
there exist some functions \pier{$\varphi \in W$, $\xi \in H$ and a subsequence of 
$\ep$} such that 
\begin{align}
&\varphi_{\ep} \to \varphi \quad \mbox{weakly in}\ W, \label{weak1} \\ 
&\beta_{\ep}(\varphi_{\ep}) \to \xi  
                                             \quad \mbox{weakly in}\ H \label{weak2} 
\end{align}
as $\ep=\ep_j\searrow 0$. 
Now we confirm that 
\begin{align}\label{semisecondequation}
\varphi - h\Delta \varphi + h\xi + h\pi(\varphi) = g. 
\end{align}
Let $\psi \in C_{\mathrm c}^{\infty}(\overline{\Omega})$. 
Then there exists a bounded domain $D \subset \Omega$ with smooth boundary 
such that $\mbox{supp}\,\psi \subset D$ 
and it follows from \eqref{ep elliptic eq} that 
\begin{align}\label{e1}
0 &= \int_{\Omega} 
          (\varphi_{\ep} - h\Delta\varphi_{\ep} 
                        + h(\beta_{\ep}(\varphi_{\ep}) + \pi(\varphi_{\ep})) - g)\psi 
\\ \notag 
&= (\varphi_{\ep}-h\Delta\varphi_{\ep}+h\beta_{\ep}(\varphi_{\ep})-g, \psi)_{H} 
     + h\int_{D}\pi(\varphi_{\ep})\psi. 
\end{align}
Here, since the embedding $H^1(D)\hookrightarrow L^2(D)$ is compact, 
we deduce from \eqref{Vesti} and \eqref{weak1} that 
\begin{align}\label{strong}
\varphi_{\ep} \to \varphi \quad\mbox{strongly in}\ L^2(D) 
\end{align}
as $\ep=\ep_j\searrow 0$.  
Thus we derive from \eqref{weak1}, \eqref{weak2}, \eqref{e1} and \eqref{strong} 
that 
\begin{align*}
\int_{\Omega} 
       (\varphi - h\Delta\varphi + h\xi +h\pi(\varphi)  - g)\psi =0  
\end{align*}
for all $\psi \in C_{\mathrm c}^{\infty}(\overline{\Omega})$, 
which implies \eqref{semisecondequation}. 

Next we show that 
\begin{align}\label{xibeta}
\xi \in \beta(\varphi) \quad \mbox{a.e.\ on}\ \Omega. 
\end{align} 
Let $E \subset \Omega$ be an arbitrary bounded domain with smooth boundary. 
Then we have  
\begin{align}\label{strongE} 
1_{E}\varphi_{\ep} \to 1_{E}\varphi \quad \mbox{strongly in}\ H 
\end{align}
as $\ep=\ep_j\searrow 0$, 
where $1_{E}$ is the characteristic function on $E$.  
Hence we see from \eqref{weak2} and \eqref{strongE} that   
\begin{align*}
\int_{\Omega} \beta_{\ep}(1_{E}\varphi_{\ep})\cdot1_{E}\varphi_{\ep} 
= (\beta_{\ep}(\varphi_{\ep}), 1_{E}\varphi_{\ep})_{H} 
\to (\xi, 1_{E}\varphi)_{H} 
= \int_{\Omega} 1_{E}\xi\cdot1_{E}\varphi 
\end{align*}
as $\ep=\ep_j\searrow 0$, 
and \pier{consequently} it holds that $1_{E}\xi \in \beta(1_{E}\varphi)$ a.e.\ on $\Omega$ 
(see e.g., \cite[Lemma 1.3, p.\ 42]{Barbu1}), 
that is, 
$$
\xi = 1_{E}\xi \in \beta(1_{E}\varphi) = \beta(\varphi) \quad \mbox{a.e.\ on}\ E. 
$$
Thus, since $E$ is arbitrary, we can obtain \eqref{xibeta}. 

Therefore combining \eqref{semisecondequation} and \eqref{xibeta} 
leads to the equation 
$\varphi - h \Delta\varphi + h(\xi + \pi(\varphi)) = g$, \pier{with} $\xi \in \beta(\varphi)$. 
Moreover, the solution \pier{$ (\varphi, \xi)$ of this problem} is unique. 
Indeed, letting $(\varphi_j, \xi_j)$, $j=1, 2$, be two solutions,   
we infer that 
\begin{align*}
0 &= \langle \varphi_1 - \varphi_2 - h\Delta(\varphi_1 - \varphi_2) 
                             + h(\xi_1 - \xi_2) + h(\pi(\varphi_1) - \pi(\varphi_2)), 
                                                             \varphi_1 - \varphi_2 \rangle_{V^{*}, V} 
\\ \notag 
&\geq (1-h\|\pi'\|_{L^{\infty}(R)})\|\varphi_1 - \varphi_2\|_{H}^2 
         + h\|\nabla(\varphi_1 - \varphi_2)\|_{H}^2 \\ \notag
&\geq \min\{1-h\|\pi'\|_{L^{\infty}(R)}, h\}\|\varphi_1 - \varphi_2\|_{V}^2,  
\end{align*}
which means that $\varphi_1 = \varphi_2$. 
Then the identity $\xi_1 = \xi_2$ holds 
by comparing the equations for $(\varphi_1, \xi_1)$ and $(\varphi_2, \xi_2)$. 
\end{proof}

\begin{prth2.1}
The problem \ref{Pn} can be written as 
\begin{equation*}\tag*{(Q)$_{n}$}\label{Qn}
     \begin{cases}
         \theta_{n+1} - h\Delta\theta_{n+1} 
         = hf_{n+1} + \ell \varphi_{n} - \ell\varphi_{n+1} + \theta_n, 
     \\[2mm] 
         \varphi_{n+1} - h\Delta\varphi_{n+1} + h(\xi_{n+1} + \pi(\varphi_{n+1})) 
         = \varphi_n + h\ell\theta_n,\ \xi_{n+1} \in \beta(\varphi_{n+1}).  
     \end{cases}
 \end{equation*}
To prove Theorem \ref{maintheorem1} 
it suffices to establish existence and uniqueness of solutions to \ref{Qn} 
in the case that $n=0$. 
Let $h \in \left(0, \frac{1}{\|\pi'\|_{L^{\infty}(\mathbb{R})}}\right)$. 
Then, by Lemma \ref{about a elliptic eq},   
there exists a unique solution $(\varphi_1, \xi_1)$ of 
$$
\varphi_1 - h\Delta \varphi_1 + h(\xi_1 + \pi(\varphi_1)) 
= \varphi_0 + \ell h\theta_0 ,\ \xi_1 \in \beta(\varphi_1),   
$$
satisfying $\varphi_1 \in W$ and $\xi_1 \in H$.  
Also, for this function $\varphi_1$ there exists a unique solution $\theta_1\in W$ 
of the equation 
$$
\theta_1 - h\Delta\theta_1 
         = hf_1 + \ell \varphi_{0} - \ell\varphi_1 + \theta_0. 
$$ 
Thus we conclude that 
there exists a unique solution $(\theta_{n+1}, \varphi_{n+1}, \xi_{n+1})$ 
of {\rm \ref{Pn}} 
satisfying $\theta_{n+1}, \varphi_{n+1} \in W$ and $\xi_{n+1} \in H$ for $n=0, ..., N-1$. 
\end{prth2.1}

\vspace{10pt}
 
\section{Uniform estimates and passage to the limit}\label{Sec4}
 
This section will prove Theorem \ref{maintheorem2}. 
We will \pier{derive} a priori estimates for \ref{Ph} 
to establish existence for \ref{P} by passing to the limit in \ref{Ph}. 

\begin{lem}\label{Firstesti}  
There exists 
$h_1 \in \left(0, \min\left\{1, \frac{1}{\|\pi'\|_{L^{\infty}(\mathbb{R})}}\right\}\right)$  
such that 
for all $\ell>0$ there \pier{is} a constant $C=C(\ell, T)>0$ such that 
\begin{align*}
&\|\overline{\theta}_h\|_{L^{\infty}(0, T; H)}^2 
 + \ell^2\|\overline{\varphi}_h\|_{L^{\infty}(0, T; V)}^2 
 + h\|\partial_t \widehat{\theta}_h\|_{L^2(0, T; H)}^2 \\ \notag 
 &+ \ell^2\|\partial_t \widehat{\varphi}_h\|_{L^2(0, T; H)}^2 
 + \|\overline{\theta}_h\|_{L^2(0, T; V)}^2 
+ \ell^2\|\widehat{\beta}(\overline{\varphi}_h)\|_{L^{\infty}(0, T; L^1(\Omega))}  
\leq C                    
\end{align*}
for all $h \in (0, h_1)$.
\end{lem}
\begin{proof}
Please note that, 
provided $h \in \left(0, \frac{1}{\|\pi'\|_{L^{\infty}(\mathbb{R})}}\right)$, 
then the solution $(\theta_{n+1}, \varphi_{n+1}, \xi_{n+1})$ to \ref{Pn} 
is well-defined due to Theorem \ref{maintheorem1}. 
Multiplying the first equation in \ref{Pn} by $h\theta_{n+1}$ and  
integrating over $\Omega$, 
we have 
\begin{align}\label{f1}
&\frac{1}{2}\int_{\Omega}|\theta_{n+1}|^2 - \frac{1}{2}\int_{\Omega}|\theta_n|^2 
+ \frac{1}{2}\int_{\Omega}|\theta_{n+1}-\theta_n|^2 
+h\int_{\Omega}|\nabla\theta_{n+1}|^2 
\\ \notag 
&=h\int_{\Omega}f_{n+1}\theta_{n+1} 
    + \ell h\int_{\Omega}\frac{\varphi_n-\varphi_{n+1}}{h}\theta_{n+1}, 
\end{align}
where the identity 
$(a-b)a = \frac{1}{2}a^2-\frac{1}{2}b^2+\frac{1}{2}(a-b)^2$ ($a, b \in \mathbb{R}$) 
was applied.  
\pier{Multiplying the second equation in \ref{Pn} by $\ell^2(\varphi_{n+1}-\varphi_n)$ 
and integrating over $\Omega$ lead to} 
\begin{align}\label{f2}
&\ell^2h\int_{\Omega}\left|\frac{\varphi_{n+1}-\varphi_n}{h}\right|^2 
+ \ell^2(\varphi_{n+1}, \varphi_{n+1}-\varphi_n)_{V}  
\\ \notag 
&+ \ell^2
     \int_{\Omega}(\xi_{n+1}+\pi(\varphi_{n+1})-\varphi_{n+1})(\varphi_{n+1}-\varphi_n)
= \ell^3 h\int_{\Omega}\theta_n \frac{\varphi_{n+1}-\varphi_n}{h}. 
\end{align}
Thus it follows from \eqref{f1}, \eqref{f2} and the Young inequality that 
\begin{align}\label{f3}
&\frac{1}{2}\int_{\Omega}|\theta_{n+1}|^2 - \frac{1}{2}\int_{\Omega}|\theta_n|^2 
+ \frac{1}{2}\int_{\Omega}|\theta_{n+1}-\theta_n|^2 
+ h\int_{\Omega}|\nabla\theta_{n+1}|^2 
\\ \notag 
&+ \ell^2 h\int_{\Omega}\left|\frac{\varphi_{n+1}-\varphi_n}{h}\right|^2 
+ \ell^2 (\varphi_{n+1}, \varphi_{n+1}-\varphi_n)_{V}  
\\ \notag 
&+ \ell^2 
     \int_{\Omega}(\xi_{n+1}+\pi(\varphi_{n+1})-\varphi_{n+1})(\varphi_{n+1}-\varphi_n)
\\ \notag 
&= h\int_{\Omega}f_{n+1}\theta_{n+1} 
    + \ell h\int_{\Omega}\frac{\varphi_n-\varphi_{n+1}}{h}\theta_{n+1} 
    + \ell^3 h\int_{\Omega}\theta_n \frac{\varphi_{n+1}-\varphi_n}{h} 
\\ \notag 
&\leq \frac{h}{2}\int_{\Omega}|f_{n+1}|^2 
         + \frac{3}{2}h\int_{\Omega}|\theta_{n+1}|^2 
         + \frac{\ell^2 h}{2}
                           \int_{\Omega}\left|\frac{\varphi_{n+1}-\varphi_n}{h}\right|^2 
         + h\ell^4\int_{\Omega}|\theta_n|^2.  
\end{align} 
Here we \pier{point out} the identity    
\begin{align}\label{f4}
\ell^2 (\varphi_{n+1}, \varphi_{n+1}-\varphi_n)_{V}  
= \frac{\ell^2 }{2}\|\varphi_{n+1}\|_{V}^2 - \frac{\ell^2 }{2}\|\varphi_n\|_{V}^2 
   + \frac{\ell^2 }{2}\|\varphi_{n+1}-\varphi_n\|_{V}^2 
\end{align} 
and 
\pier{recall} the inclusion $\xi_{n+1} \in \beta(\varphi_{n+1})$, (A1), 
the definition of the subdifferential, and the Young inequality \pier{in order to infer} that  
\begin{align}\label{f5}
&\ell^2 
     \int_{\Omega}(\xi_{n+1}+\pi(\varphi_{n+1})-\varphi_{n+1})(\varphi_{n+1}-\varphi_n)
\\ \notag
&= \ell^2 \int_{\Omega}\xi_{n+1}(\varphi_{n+1}-\varphi_n) 
   + \ell^2 \int_{\Omega}\pi(\varphi_{n+1})(\varphi_{n+1}-\varphi_n) 
   - \ell^2 \int_{\Omega}\varphi_{n+1}(\varphi_{n+1}-\varphi_n) 
\\ \notag
&\geq \ell^2 \int_{\Omega}\widehat{\beta}(\varphi_{n+1}) 
         - \ell^2 \int_{\Omega}\widehat{\beta}(\varphi_n) 
         - 2(\|\pi'\|_{L^{\infty}(\mathbb{R})}^2+1)\ell^2 h \|\varphi_{n+1}\|_{V}^2
\\ \notag 
&\,\quad-\frac{\ell^2 h}{4}
                        \int_{\Omega}\left|\frac{\varphi_{n+1}-\varphi_n}{h}\right|^2. 
\end{align}
Thus from \eqref{f3}-\eqref{f5} we have 
\begin{align}\label{f6}
&\frac{1}{2}\int_{\Omega}|\theta_{n+1}|^2 - \frac{1}{2}\int_{\Omega}|\theta_n|^2 
+ \frac{1}{2}\int_{\Omega}|\theta_{n+1}-\theta_n|^2 
+h\int_{\Omega}|\nabla\theta_{n+1}|^2 
\\ \notag 
&+ \frac{\ell^2 h}{4}\int_{\Omega}\left|\frac{\varphi_{n+1}-\varphi_n}{h}\right|^2 
   + \frac{\ell^2}{2}\|\varphi_{n+1}\|_{V}^2 - \frac{\ell^2}{2}\|\varphi_n\|_{V}^2 
\\ \notag
&+ \frac{\ell^2}{2}\|\varphi_{n+1}-\varphi_n\|_{V}^2 
  + \ell^2 \int_{\Omega}\widehat{\beta}(\varphi_{n+1}) 
  - \ell^2 \int_{\Omega}\widehat{\beta}(\varphi_n) 
\\ \notag 
&\leq \frac{h}{2}\int_{\Omega}|f_{n+1}|^2 
         + \frac{3}{2}h\int_{\Omega}|\theta_{n+1}|^2 
         + h\ell^4 \int_{\Omega}|\theta_n|^2 
+ 2(\|\pi'\|_{L^{\infty}(\mathbb{R})}^2+1)\ell^2 h \|\varphi_{n+1}\|_{V}^2. 
\end{align}
\pier{Therefore}, summing \eqref{f6} over $n=0, ..., m-1$ with $1\leq m \leq N$, 
we obtain the inequality   
\begin{align*}
&\frac{1}{2}(1-3h)\int_{\Omega}|\theta_m|^2 
- \frac{1}{2}\int_{\Omega}|\theta_0|^2 
+ \frac{1}{2}\sum_{n=0}^{m-1} \int_{\Omega}|\theta_{n+1}-\theta_n|^2 
+ h\sum_{n=0}^{m-1} \int_{\Omega}|\nabla\theta_{n+1}|^2 
\\ \notag 
&+ \frac{\ell^2 h}{4}\sum_{n=0}^{m-1} \int_{\Omega}
                                                  \left|\frac{\varphi_{n+1}-\varphi_n}{h}\right|^2 
+ \frac{\ell^2}{2}(1 - 4(\|\pi'\|_{L^{\infty}(\mathbb{R})}^2+1)h)\|\varphi_m\|_{V}^2 
- \frac{\ell^2}{2}\|\varphi_0\|_{V}^2 
\\ \notag 
&+ \frac{\ell^2}{2}\sum_{n=0}^{m-1}\|\varphi_{n+1}-\varphi_n\|_{V}^2 
           + \ell^2 \int_{\Omega}\widehat{\beta}(\varphi_m) 
           - \ell^2 \int_{\Omega}\widehat{\beta}(\varphi_0) 
\\ \notag 
&\leq \frac{h}{2}\sum_{n=0}^{m-1}\int_{\Omega}|f_{n+1}|^2 
         + \frac{3}{2}h\sum_{n=0}^{m-2}\int_{\Omega}|\theta_{n+1}|^2 
\\ \notag
         &\,\quad+ h\ell^4\sum_{n=0}^{m-1}\int_{\Omega}|\theta_n|^2 
         + 2(\|\pi'\|_{L^{\infty}(\mathbb{R})}^2+1)\ell^2 h
                                               \sum_{n=0}^{m-2}\|\varphi_{n+1}\|_{V}^2.  
\end{align*}
\pier{Then} there exists 
$h_1 \in \left(0, \min\left\{1, \frac{1}{\|\pi'\|_{L^{\infty}(\mathbb{R})}}\right\}\right)$ 
(e.g., $h_1 = \frac{1}{4(\|\pi'\|_{L^{\infty}(\mathbb{R})}^2+1)}$) 
such that 
for all $\ell > 0$ 
there exists a constant $C_1 = C_{1}(\ell, T) > 0$ such that 
\begin{align*}
&\int_{\Omega}|\theta_m|^2 
+ h^2\sum_{n=0}^{m-1} \int_{\Omega}|\delta\theta_n|^2 
+ h\sum_{n=0}^{m-1} \int_{\Omega}|\nabla\theta_{n+1}|^2 
+ \ell^2 h\sum_{n=0}^{m-1} \int_{\Omega}|\delta\varphi_n|^2 
+ \ell^2 \|\varphi_m\|_{V}^2 
\\ \notag 
&+ \ell^2 \int_{\Omega}\widehat{\beta}(\varphi_m) 
\leq C_1 + C_1 h\sum_{n=0}^{m-1}\int_{\Omega}|\theta_j|^2 
         + C_1 h \sum_{n=0}^{m-1}\|\varphi_j\|_{V}^2  
\end{align*}
for all $h \in (0, h_1)$ and $m=1,..., N$.  
Hence we see from 
the discrete Gronwall lemma (see e.g., \cite[Prop.\ 2.2.1]{Jerome}) that  
for all $\ell > 0$ 
there exists a constant $C_2=C_2(\ell, T)>0$ 
such that 
\begin{align*}
&\int_{\Omega}|\theta_m|^2 
+ h^2\sum_{n=0}^{m-1} \int_{\Omega}|\delta\theta_n|^2 
+ h\sum_{n=0}^{m-1} \int_{\Omega}|\nabla\theta_{n+1}|^2 
\\ \notag 
&+ \ell^2 h\sum_{n=0}^{m-1} \int_{\Omega}|\delta\varphi_n|^2 
+ \ell^2 \|\varphi_m\|_{V}^2 
+ \ell^2 \int_{\Omega}\widehat{\beta}(\varphi_m)  
\leq C_2
\end{align*}
for all $h \in (0, h_1)$ and $m=1, ..., N$.
\end{proof}

\begin{lem}\label{Secondesti}
Let $h_1$ be as in Lemma \ref{Firstesti}. 
Then for all $\ell > 0$ 
there exists a constant $C=C(\ell, T)>0$ such that 
\begin{align*}
 \|\partial_t \widehat{\theta}_h\|_{L^2(0, T; H)}^2 
 + \|\overline{\theta}_h\|_{L^{\infty}(0, T; V)}^2
 \leq C                    
\end{align*}
for all $h \in (0, h_1)$. 
\end{lem}
\begin{proof}
Multiplying the first equation in \ref{Pn} by $h\delta_h \theta_n$ and 
integrating over $\Omega$, we have 
\begin{align}\label{s1}
h\int_{\Omega}|\delta_h \theta_n|^2 
+ \ell h \int_{\Omega} \delta_h \varphi_n\cdot\delta_h \theta_n 
+ h\int_{\Omega}\nabla\theta_{n+1}\cdot\nabla\delta_h\theta_n 
= h\int_{\Omega}f_{n+1}\delta_h \theta_n. 
\end{align}
Here it holds that 
\begin{align}\label{s2}
h\int_{\Omega}\nabla\theta_{n+1}\cdot\nabla\delta_h\theta_n  
= \frac{1}{2}\int_{\Omega}|\nabla\theta_{n+1}|^2 
   - \frac{1}{2}\int_{\Omega}|\nabla\theta_n|^2 
   + \frac{1}{2}\int_{\Omega}|\nabla(\theta_{n+1}-\theta_n)|^2. 
\end{align}
Thus we see from \eqref{s1}, \eqref{s2} and the Young inequality that  
\begin{align}\label{s3}
&h\int_{\Omega}|\delta_h \theta_n|^2 
+ \frac{1}{2}\int_{\Omega}|\nabla\theta_{n+1}|^2 
   - \frac{1}{2}\int_{\Omega}|\nabla\theta_n|^2 
   + \frac{1}{2}\int_{\Omega}|\nabla(\theta_{n+1}-\theta_n)|^2 
\\ \notag 
&= h\int_{\Omega}f_{n+1}\delta_h \theta_n 
     - \ell h \int_{\Omega} \delta_h \varphi_n\cdot\delta_h \theta_n 
\\ \notag
&\leq h\int_{\Omega}|f_{n+1}|^2 + \frac{h}{2}\int_{\Omega}|\delta_h \theta_n|^2 
         + \ell^2 h \int_{\Omega} |\delta_h \varphi_n|^2.  
\end{align}
Therefore, by summing \eqref{s3} over $n=0, ..., m-1$ with $1\leq m \leq N$  
and Lemma \ref{Firstesti}, we can prove Lemma \ref{Secondesti}. 
\end{proof}

\begin{lem}\label{Thirdesti}
Let $h_1$ be as in Lemma \ref{Firstesti}. 
Then for all $\ell > 0$
there exists a constant $C=C(\ell, T)>0$ such that 
$$
\|\overline{\xi}_h\|_{L^2(0, T; H)}^2 \leq C
$$ 
for all $h \in (0, h_1)$. 
\end{lem}
\begin{proof}
We formally derive the estimate by assuming that 
$\beta$ is Lipschitz continuous (as for the Yosida approximation) 
and observing that 
the estimate can be extended to the general case by lower semicontinuity. 
We test the second equation in \ref{Pn} by $h\xi_{n+1}$.     
Since 
$(-\Delta\varphi_{n+1}, \xi_{n+1})_{H} 
= \int_{\Omega}\beta'(\varphi_{n+1})|\nabla\varphi_{n+1}|^2\geq0$, 
then by the Young inequality we obtain 
\begin{align*}
&h\int_{\Omega}|\xi_{n+1}|^2 
\\ \notag 
&\leq h\ell\int_{\Omega}\theta_n \xi_{n+1}
   - h\int_{\Omega}\delta_h \varphi_n\cdot\xi_{n+1}
   - h\int_{\Omega}\pi(\varphi_{n+1})\xi_{n+1}
\\ \notag 
&\leq \frac{3}{2}h\ell^2\int_{\Omega}|\theta_n|^2 
         + \frac{3}{2}h\int_{\Omega}|\delta_h \varphi_n|^2 
         + \frac{3}{2}\|\pi'\|_{L^{\infty}(\mathbb{R})}^2h\int_{\Omega}|\varphi_{n+1}|^2 
         + \frac{1}{2}h\int_{\Omega}|\xi_{n+1}|^2.
\end{align*}
Thus, by summing over $n=0, ..., m-1$ with $1\leq m \leq N$, \pier{and recalling}
Lemmas~\ref{Firstesti} and \ref{Secondesti}, 
we conclude that Lemma \ref{Thirdesti} holds.   
\end{proof}

\begin{lem}\label{Fourthesti}
Let $h_1$ be as in Lemma \ref{Firstesti}. 
Then for all $\ell > 0$
there exist constants 
$C=C(\ell, T)>0$ and $\widetilde{C}=\widetilde{C}(\ell, T)>0$ such that 
\begin{align}
&\|\Delta\overline{\theta}_h\|_{L^2(0, T; H)}^2 
+ \|\Delta\overline{\varphi}_h\|_{L^2(0, T; H)}^2
\leq C, \notag 
\\ 
&\|\overline{\theta}_h\|_{L^2(0, T; W)}^2 
+ \|\overline{\varphi}_h\|_{L^2(0, T; W)}^2
\leq \widetilde{C}   \label{secondineqinLem4.4}                  
\end{align}
for all $h \in (0, h_1)$. 
\end{lem}
\begin{proof}
From the first equation in \ref{Pn} we have 
\begin{align*}
h^{1/2}\|\Delta\theta_{n+1}\|_{H}  
&= h^{1/2}\|\delta_h \theta_n + \ell\delta_h \varphi_n - f_{n+1}\|_{H}
\\ \notag 
&\leq h^{1/2}\|\delta_h \theta_n\|_{H} + h^{1/2}\ell\|\delta_h \varphi_n\|_{H} 
         + h^{1/2}\|f_{n+1}\|_{H}.
\end{align*}
Thus, by Lemmas \ref{Firstesti} and \ref{Secondesti},  
there exists a constant $C_1=C_1(T)>0$ such that 
\begin{align}\label{fo1}
 \|\Delta\overline{\theta}_h\|_{L^2(0, T; H)}^2 \leq C_1.                     
\end{align}
It follows from the second equation in \ref{Pn} that  
\begin{align*}
h^{1/2}\|\Delta\varphi_{n+1}\|_{H} 
&= h^{1/2}\|\delta_h \varphi_n + \xi_{n+1} + \pi(\varphi_{n+1}) - \ell\theta_n\|_{H}
\\ \notag 
&\leq h^{1/2}\|\delta_h \varphi_n\|_{H} 
        + h^{1/2}\|\xi_{n+1}\|_{H} 
        + h^{1/2}\|\pi'\|_{L^{\infty}(\mathbb{R})}\|\varphi_{n+1}\|_{H} 
        + h^{1/2}\ell\|\theta_n\|_{H}.
\end{align*}
Hence Lemmas \pier{\ref{Firstesti}-\ref{Thirdesti}} 
yield that there exists a constant $C_2=C_2(T)>0$ such that 
\begin{align}\label{fo2}
 \|\Delta\overline{\varphi}_h\|_{L^2(0, T; H)}^2 \leq C_2.                     
\end{align}
Now we recall Lemma \ref{Firstesti} once more 
and invoke the elliptic regularity theory 
to infer the estimate \eqref{secondineqinLem4.4}. 
Therefore, by virtue of \eqref{fo1} and \eqref{fo2}, 
Lemma \ref{Fourthesti} is completely proved.  
\end{proof}

\begin{lem}\label{Fifthesti}
Let $h_1$ be as in Lemma \ref{Firstesti}. 
Then for all $\ell > 0$ 
there exists a constant $C=C(\ell, T)>0$ such that 
\begin{align*}
 \|\widehat{\theta}_h\|_{H^1(0, T; H) \cap L^{\infty}(0, T; V)} 
 + \|\widehat{\varphi}_h\|_{H^1(0, T; H) \cap L^{\infty}(0, T; V)} 
 \leq C                    
\end{align*}
for all $h \in (0, h_1)$. 
\end{lem}
\begin{proof}
Lemmas \ref{Firstesti} \pier{and \ref{Secondesti}, along with \eqref{rem1}-\eqref{rem4},}
lead to Lemma \ref{Fifthesti}.
\end{proof}

\begin{prth2.2}
By Lemmas \ref{Firstesti}-\ref{Fifthesti}, 
\eqref{deltah}-\eqref{overandunderline},  \eqref{rem5} and \eqref{rem6} 
there exist some functions 
$\theta \in H^1(0, T; H) \cap L^{\infty}(0, T; V) \cap L^2(0, T; W)$, 
$\varphi \in H^1(0, T; H) \cap L^{\infty}(0, T; V) \cap L^2(0, T; W)$ 
and  $\xi \in L^2(0, T; H)$ \pier{and a subsequence of $h$} such that 
\begin{align}
&\widehat{\theta}_h \to \theta 
\quad \mbox{weakly in}\ H^1(0, T; H), \label{w1} \\ 
&\widehat{\varphi}_h \to \varphi  
\quad \mbox{weakly in}\ H^1(0, T; H), \label{w2} \\ \notag 
&\overline{\theta}_h \to \theta 
\quad \mbox{weakly$^{*}$ in}\ L^{\infty}(0, T; V), \\ \notag 
&\overline{\varphi}_h \to \varphi 
\quad \mbox{weakly$^{*}$ in}\ L^{\infty}(0, T; V), \\ 
&\overline{\theta}_h \to \theta 
\quad \mbox{weakly in}\ L^2(0, T; W), \label{w3} \\ 
&\overline{\varphi}_h \to \varphi 
\quad \mbox{weakly in}\ L^2(0, T; W), \label{w4} \\ 
&\underline{\theta}_h = \overline{\theta}_h - h\partial_t \widehat{\theta}_h 
\to \theta 
\quad \mbox{weakly in}\ L^2(0, T; H), \label{w5} \\ 
&\overline{\xi}_h \to \xi 
\quad \mbox{weakly in}\ L^2(0, T; H) \label{w6}  
\end{align}
as $h= h_j\searrow0$. 
Combining \eqref{w1}-\eqref{w3},  
observing that $\overline{f}_h \to f$ strongly in $L^2(0, T; H)$ 
as $h\searrow0$ 
(cf. Remark \ref{about.conv.of.fh}) 
and passing to the limit 
in the first equation in \ref{Ph}  
lead to \eqref{df1}.    
Now we show that 
\begin{align}\label{collikurima1}
\partial_t \varphi -\Delta\varphi + \xi + \pi(\varphi) = \ell\theta. 
\end{align}
Let $\psi \in C_{\mathrm c}^{\infty}([0, T]\times\overline{\Omega})$. 
Then there exists a bounded domain $D \subset \Omega$ with smooth boundary 
such that $\mbox{supp}\,\psi \subset (0, T)\times D$ 
and we see from the second equation in \ref{Ph} that 
\begin{align}\label{collikurima2}
0 &= \int_{0}^{T}
      (\partial_t \widehat{\varphi}_h (t) - \Delta\overline{\varphi}_h (t) 
           + \overline{\xi}_h (t) - \ell\underline{\theta}_h (t), \psi(t))_{H}\,dt 
\\ \notag 
     &\,\quad+ \int_{0}^{T}\left(\int_{D}\pi(\overline{\varphi}_h (t))\psi(t)\right)\,dt. 
\end{align}
Here, since the embedding $H^1(D)\hookrightarrow L^2(D)$ is compact, 
we derive from Lemma \ref{Fifthesti} and \eqref{w2} that 
\begin{align}\label{collikurima3}
\widehat{\varphi}_h \to \varphi \quad\mbox{strongly in}\ C([0, T]; L^2(D))  
\end{align}
as $h=h_j\searrow0$ (see e.g., \cite[Section 8, Corollary 4]{Simon}). 
Also,  
we infer from \eqref{rem6}, Lemma~\ref{Firstesti} and \eqref{collikurima3} that
\begin{align*}
\|\overline{\varphi}_h - \varphi\|_{L^2(0, T; L^2(D))} 
&\leq \|\overline{\varphi}_h - \widehat{\varphi}_h\|_{L^2(0, T; H)} 
       + \|\widehat{\varphi}_h - \varphi\|_{L^2(0, T; L^2(D))} 
\\    
&= \frac{h}{\sqrt{3}}\|\partial_t \widehat{\varphi}_h\|_{L^2(0, T; H)} 
     + \|\widehat{\varphi}_h - \varphi\|_{L^2(0, T; L^2(D))} 
\to 0 
\end{align*}
as $h=h_j\searrow0$,
\pier{whence} it holds that  
\begin{align}\label{collikurima4}
\overline{\varphi}_h \to \varphi \quad\mbox{strongly in}\ L^2(0, T; L^2(D))
\end{align}
as $h=h_j\searrow0$. 
Thus it follows from \eqref{w2}, \eqref{w4}-\eqref{w6}, 
\eqref{collikurima2} and \eqref{collikurima4}  
that 
\begin{align*}
\int_{0}^{T}\left(\int_{\Omega} 
         (\partial_t \varphi(t) - \Delta\varphi(t) 
                                + \xi(t) + \pi(\varphi(t)) - \ell\theta(t))\psi(t) \right)\,dt 
=0 
\end{align*}
for all $\psi \in C_{\mathrm c}^{\infty}([0, T]\times\overline{\Omega})$, 
which implies \eqref{collikurima1}. 

Next we prove that 
\begin{align}\label{tomato1}
\xi \in \beta(\varphi) \quad \mbox{a.e.\ on}\ \Omega\times(0, T). 
\end{align} 
Let $E \subset \Omega$ be an arbitrary bounded domain with smooth boundary. 
Then we can verify that   
\begin{align}\label{tomato2} 
1_{E}\overline{\varphi}_h \to 1_{E}\varphi \quad\mbox{strongly in}\ L^2(0, T; H)
\end{align}
as $h=h_j\searrow0$, 
where $1_{E}$ is the characteristic function on $E$.    
Hence from \eqref{w6} and \eqref{tomato2}
we deduce that     
\begin{align*}
\int_{0}^{T} (1_{E}\overline{\xi}_h (t), 1_{E}\overline{\varphi}_h (t))_{H}\,dt  
= \int_{0}^{T} (\overline{\xi}_h (t), 1_{E}\overline{\varphi}_h (t))_{H}\,dt 
&\to \int_{0}^{T} (\xi(t), 1_{E}\overline{\varphi}_h (t))_{H}\,dt 
\\ \notag 
&= \int_{0}^{T} (1_{E}\xi(t), 1_{E}\overline{\varphi}_h (t))_{H}\,dt
\end{align*}
as $h=h_j\searrow0$.  
Then the inclusion  
$1_{E}\xi \in \beta(1_{E}\varphi)$ holds a.e.\ on $\Omega\times(0, T)$   
(see e.g., \cite[Lemma 1.3, p.\ 42]{Barbu1}), 
that is, 
$$
\xi = 1_{E}\xi \in \beta(1_{E}\varphi) = \beta(\varphi) 
\quad \mbox{a.e.\ on}\ E\times(0, T). 
$$
Thus, since $E \subset \Omega$ is arbitrary,  
we conclude that \eqref{tomato1} holds.   

Therefore combining \eqref{collikurima1} and \eqref{tomato1} leads to \eqref{df2}. 
Next we confirm \eqref{df3}. 
Let $E \subset \Omega$ be an arbitrary bounded domain with smooth boundary. 
Then we see that 
\begin{align*}
\widehat{\theta}_h \to \theta ,\  
\widehat{\varphi}_h \to \varphi 
\quad\mbox{strongly in}\ C([0, T]; L^2(E)). 
\end{align*} 
In particular, it follows from \eqref{hat} that 
\begin{align*}
&\|\theta_0 - \theta(0)\|_{L^2(E)} 
= \|\widehat{\theta}_h (0) - \theta(0)\|_{L^2(E)} \to 0, 
\\ 
&\|\varphi_0 - \varphi(0)\|_{L^2(E)} 
= \|\widehat{\varphi}_h (0) - \varphi(0)\|_{L^2(E)} \to 0
\end{align*}
as $h=h_j\searrow0$. 
Hence we can show that 
\begin{align*}
\theta(0) = \theta_0,\ \varphi(0)=\varphi_0                                          
         \quad\mbox{a.e.\ on}\ E.  
\end{align*}
Hence, since $E \subset \Omega$ is arbitrary, 
we can obtain \eqref{df3}. 

Next, we should prove the uniqueness of the solution,
\pier{which is of course known. However, for the sake of completeness
we detail here a uniqueness proof for the reader.}
Let $(\theta_j, \varphi_j, \xi_j)$, $j=1, 2$, be two solutions. 
Then the identities 
\begin{align}
&\partial_t (\theta_1 - \theta_2) + \ell\partial_t (\varphi_1 - \varphi_2) 
- \Delta(\theta_1 - \theta_2) 
= 0, \label{milan1} 
\\ 
&\partial_t (\varphi_1 - \varphi_2) - \Delta(\varphi_1 - \varphi_2) 
= \ell(\theta_1 - \theta_2) -\xi_1+\xi_2-\pi(\varphi_1)+\pi(\varphi_2)  \label{milan2} 
\end{align}   
hold a.e.\ on $\Omega\times(0, T)$. 
Integrating \eqref{milan1} over $(0, t)$, where $t \in [0, T]$, 
multiplying by $\theta_1 - \theta_2$ and integrating over $\Omega$, 
we have 
\begin{align}\label{milan3}
&\|\theta_1(t) - \theta_2(t)\|_{H}^2 
+ \ell(\varphi_1(t)-\varphi_2(t), \theta_1(t) - \theta_2(t))_{H} 
\\ \notag 
&+ \frac{1}{2}\frac{d}{dt}
                       \left\|\nabla\int_{0}^{t}(\theta_1(s) - \theta_2(s))\,ds\right\|_{H}^2 
= 0. 
\end{align} 
On the other hand, 
multiplying \eqref{milan2} by $\varphi_1 - \varphi_2$ 
leads to the identity 
\begin{align}\label{milan4}
&\frac{1}{2}\frac{d}{dt}\|\varphi_1(t)-\varphi_2(t)\|_{H}^2 
+ \|\nabla(\varphi_1(t)-\varphi_2(t))\|_{H}^2 
\\ \notag 
&= \ell(\theta_1(t)-\theta_2(t), \varphi_1(t)-\varphi_2(t))_{H} 
   - (\xi_1(t)-\xi_2(t), \varphi_1(t)-\varphi_2(t))_{H} 
\\ \notag  
   &\,\quad- (\pi(\varphi_1(t))-\pi(\varphi_2(t)), \varphi_1(t)-\varphi_2(t))_{H}. 
\end{align} 
Thus it follows from \eqref{milan3}, \eqref{milan4} 
and the monotonicity of $\beta$ 
that 
\begin{align}\label{milan5}
&\|\theta_1(t) - \theta_2(t)\|_{H}^2 
+ \frac{1}{2}\frac{d}{dt}
                       \left\|\nabla\int_{0}^{t}(\theta_1(s) - \theta_2(s))\,ds\right\|_{H}^2 
\\ \notag  
&+ \frac{1}{2}\frac{d}{dt}\|\varphi_1(t)-\varphi_2(t)\|_{H}^2 
+ \|\nabla(\varphi_1(t)-\varphi_2(t))\|_{H}^2  
\\ \notag 
&= - (\xi_1(t)-\xi_2(t), \varphi_1(t)-\varphi_2(t))_{H} 
     - (\pi(\varphi_1(t))-\pi(\varphi_2(t)), \varphi_1(t)-\varphi_2(t))_{H} 
\\ \notag 
&\leq \|\pi'\|_{L^{\infty}(\mathbb{R})}\|\varphi_1(t)-\varphi_2(t)\|_{H}^2. 
\end{align}
Owing to the initial conditions \eqref{df3} satisfied 
by both $(\theta_j, \varphi_j)$, $j=1, 2$,  
the integration of \eqref{milan5} over $(0, t)$, where $t \in [0, T]$, 
yields that  
\begin{align*}
&\int_{0}^{t} \|\theta_1(s) - \theta_2(s)\|_{H}^2\,ds 
+ \frac{1}{2}\left\|\nabla\int_{0}^{t}(\theta_1(s) - \theta_2(s))\,ds\right\|_{H}^2 
\\ \notag  
&+ \frac{1}{2}\|\varphi_1(t)-\varphi_2(t)\|_{H}^2 
+ \int_{0}^{t}\|\nabla(\varphi_1(s)-\varphi_2(s))\|_{H}^2\,ds   
\\ \notag 
&\leq \|\pi'\|_{L^{\infty}(\mathbb{R})}
                                          \int_{0}^{t}\|\varphi_1(s)-\varphi_2(s)\|_{H}^2\,ds.  
\end{align*}
Therefore, applying the Gronwall lemma, 
we see that $\theta_1=\theta_2$ and $\varphi_1=\varphi_2$. 
Then the identity $\xi_1=\xi_2$ holds by \eqref{df2}.  
\qed 
\end{prth2.2}

 \section{Error estimates}\label{Sec5}
 
In this section we will prove Theorem \ref{erroresti}.

\begin{lem}\label{semierroresti} 
Let $h_1$ be as in Lemma \ref{Firstesti}. 
Then for all $\ell >0$ 
there exists a constant $M_1=M_1(\ell, T)>0$ such that 
\begin{align}\label{semierroresti1}
&\ell\|\widehat{\varphi}_h - \varphi \|_{L^{\infty}(0, T; H)} 
+ \ell\|\overline{\varphi}_h - \varphi\|_{L^2(0, T; V)} \\ \notag
&+ \|\widehat{\theta}_h - \theta 
                                  + \ell(\widehat{\varphi}_h - \varphi)\|_{L^{\infty}(0, T; H)} 
+ \|\overline{\theta}_h - \theta\|_{L^2(0, T; V)}  \\ \notag 
&\leq M_1 h^{1/2} + M_1 \|\overline{f}_h - f\|_{L^2(0, T; H)} 
\end{align}
for all $h \in (0, h_1)$.
In particular, 
for all $\ell > 0$ 
there exists a constant $M_2=M_2(\ell, T)>0$ such that 
\begin{align*}
 \|\widehat{\theta}_h - \theta \|_{L^{\infty}(0, T; H)} 
 \leq M_2 h^{1/2} + M_2 \|\overline{f}_h - f\|_{L^2(0, T; H)} 
\end{align*} 
for all $h \in (0, h_1)$. 
\end{lem} 
\begin{proof}
Let $h \in (0, h_1)$. 
Then we infer from the second equation in \ref{Ph}, 
\eqref{deltah}-\eqref{overandunderline} and \eqref{df2} that 
\begin{align*}
&\frac{\ell^2}{2}\frac{d}{dt}\|\widehat{\varphi}_h (t) - \varphi(t)\|_{H}^2 
\\ \notag
&= - \ell^2\|\nabla(\overline{\varphi}_h(t)-\varphi(t))\|_{H}^2 
   - \ell^2(-\Delta(\overline{\varphi}_h(t)-\varphi(t)), 
                       \widehat{\varphi}_h(t)-\overline{\varphi}_h(t))_H 
\\ \notag  
&\,\quad- \ell^2(\overline{\xi}_h(t)+\pi(\overline{\varphi}_h(t))-\xi(t)-\pi(\varphi(t)), 
                                                                \overline{\varphi}_h(t)-\varphi(t))_H 
\\ \notag 
&\,\quad-\ell^2(\overline{\xi}_h(t)+\pi(\overline{\varphi}_h(t))-\xi(t)-\pi(\varphi(t)), 
                                             \widehat{\varphi}_h(t)-\overline{\varphi}_h(t))_H 
\\ \notag
&\,\quad+ \ell^3(\widehat{\theta}_h(t)-\theta(t)
                                      +\ell(\widehat{\varphi}_h(t)-\varphi(t)), 
                                                                \widehat{\varphi}_h(t)-\varphi(t))_H 
       -\ell^4\|\widehat{\varphi}_h(t)-\varphi(t)\|_H^2 
\\ \notag  
&\,\quad +\ell^3(\overline{\theta}_h(t)-\widehat{\theta}_h(t), 
                                                                 \widehat{\varphi}_h(t)-\varphi(t))_H 
- \ell^3 h (\partial_t \widehat{\theta}_h(t), \widehat{\varphi}_h(t)-\varphi(t))_H\pier{.} 
\end{align*}  
\pier{Hence, integrating over $(0, t)$, where $t \in [0, T]$, the Schwarz 
and Young inequalities help us to deduce that}
\begin{align}\label{p1}
&\frac{\ell^2}{2}\|\widehat{\varphi}_h (t) - \varphi(t)\|_{H}^2 
+ \ell^2 \int_0^t\|\nabla(\overline{\varphi}_h(s)-\varphi(s))\|_H^2\,ds 
\\ \notag 
&\leq \ell^2 \|\Delta(\overline{\varphi}_h-\varphi)\|_{L^2(0, T; H)}
                                   \|\widehat{\varphi}_h-\overline{\varphi}_h\|_{L^2(0, T; H)} 
\\ \notag 
&\,\quad-\ell^2 \int_{0}^{t} 
                 (\overline{\xi}_h(s)+\pi(\overline{\varphi}_h(s))-\xi(s)-\pi(\varphi(s)), 
                                                          \overline{\varphi}_h(s)-\varphi(s))_H\,ds  
\\ \notag 
&\,\quad+\ell^2 
          \|\overline{\xi}_h+\pi(\overline{\varphi}_h)-\xi-\pi(\varphi)\|_{L^2(0, T; H)}
                                   \|\widehat{\varphi}_h-\overline{\varphi}_h \|_{L^2(0, T; H)} 
\\ \notag 
&\,\quad+\frac{3}{2}\ell^2 
   \int_0^t
   \|\widehat{\theta}_h(s)-\theta(s)+\ell(\widehat{\varphi}_h(s)-\varphi(s))\|_H^2\,ds 
\\ \notag 
&\,\quad-\frac{2}{3}\ell^4 \int_{0}^{t}\|\widehat{\varphi}_h (s) - \varphi(s)\|_H^2\,ds 
+\frac{3}{2}\ell^2 \|\overline{\theta}_h - \widehat{\theta}_h\|_{L^2(0, T; H)}^2 
+\frac{3}{2}\ell^2 h^2 \|\partial_t \widehat{\theta}_h\|_{L^2(0, T; H)}^2.   
\end{align} 
Also, the first equation in \ref{Ph} and \eqref{df1} yield that 
\begin{align*}
&\frac{1}{2}\frac{d}{dt}
         \|\widehat{\theta}_h(t)-\theta(t)+\ell(\widehat{\varphi}_h(t)-\varphi(t))\|_H^2
\\ \notag 
&= -\|\nabla(\overline{\theta}_h(t)-\theta(t))\|_H^2 
    - (-\Delta(\overline{\theta}_h(t)-\theta(t)), 
                                                  \widehat{\theta}_h(t)-\overline{\theta}_h(t))_H 
\\ \notag 
&\,\quad-\ell(-\Delta(\overline{\theta}_h(t)-\theta(t)), 
                                                \widehat{\varphi}_h(t)-\overline{\varphi}_h(t))_H 
    -\ell(\nabla(\overline{\theta}_h(t)-\theta(t)), 
                                                     \nabla(\overline{\varphi}_h(t)-\varphi(t)))_H 
\\ \notag 
&\,\quad+ (\overline{f}_h (t) - f(t), 
                \widehat{\theta}_h(t)-\theta(t)+\ell(\widehat{\varphi}_h(t)-\varphi(t)))_H.
\end{align*}
Thus\pier{, from the integration over $(0, t)$, where $t \in [0, T]$, 
and the Schwarz \pier{and Young inequalities} we see} that 
\begin{align}\label{p2}
&\frac{1}{2}
\|\widehat{\theta}_h(t)-\theta(t)+\ell(\widehat{\varphi}_h(t)-\varphi(t))\|_H^2 
+ \frac{1}{2}\int_{0}^{t} \|\nabla(\overline{\theta}_h(s)-\theta(s))\|_{H}^2\,ds 
\\ \notag 
&\leq \|\Delta(\overline{\theta}_h-\theta)\|_{L^2(0, T; H)}
                                   \|\widehat{\theta}_h-\overline{\theta}_h\|_{L^2(0, T; H)} 
\\ \notag 
&\,\quad+ \ell\|\Delta(\overline{\theta}_h-\theta)\|_{L^2(0, T; H)}
                                   \|\widehat{\varphi}_h-\overline{\varphi}_h\|_{L^2(0, T; H)} 
+ \frac{\ell^2}{2}
                              \int_0^t\|\nabla(\overline{\varphi}_h(s)-\varphi(s))\|_H^2\,ds 
\\ \notag 
&\,\quad+ \frac{1}{2}\|\overline{f}_h - f\|_{L^2(0, T; H)}^2  
+ \frac{1}{2}\int_{0}^{t} 
   \|\widehat{\theta}_h(s)-\theta(s)+\ell(\widehat{\varphi}_h(s)-\varphi(s))\|_H^2\,ds.   
\end{align} 
Here, 
thanks to the inclusions 
$\overline{\xi}_h (t) \in \beta(\overline{\varphi}_h(t))$, 
$\xi(t) \in \beta(\varphi(t))$ 
and the monotonicity of $\beta$, 
it holds that  
\begin{align}\label{p3}
&- \ell^2(\overline{\xi}_h(t)+\pi(\overline{\varphi}_h(t))-\xi(t)-\pi(\varphi(t), 
                                                       \overline{\varphi}_h(t)-\varphi(t))_H 
\\ \notag 
&= - \ell^2(\overline{\xi}_h (t)-\xi(t), \overline{\varphi}_h(t)-\varphi(t))_H 
   - \ell^2(\pi(\overline{\varphi}_h(t))-\pi(\varphi(t)), 
                                                    \overline{\varphi}_h(t)-\varphi(t))_H 
\\ \notag 
&\leq \ell^2\|\pi'\|_{L^{\infty}(\mathbb{R})}\|\overline{\varphi}_h(t)-\varphi(t)\|_{H}^2  
\\ \notag 
&\leq 2\ell^2\|\pi'\|_{L^{\infty}(\mathbb{R})}
                                        \|\overline{\varphi}_h(t)-\widehat{\varphi}_h(t)\|_{H}^2 
        + 2\ell^2\|\pi'\|_{L^{\infty}(\mathbb{R})}
                                        \|\widehat{\varphi}_h(t)-\varphi(t)\|_{H}^2. 
\end{align}
Hence, 
in view of \eqref{hat}, \eqref{rem5}, \eqref{rem6}, \eqref{df3}   
and Lemmas \ref{Firstesti}-\ref{Fourthesti}, 
we collect \eqref{p1}-\eqref{p3} and deduce that   
there exists a constant $C_1=C_1(\ell, T)>0$ such that 
\begin{align*}
&\frac{\ell^2}{2}\|\widehat{\varphi}_h (t) - \varphi(t)\|_{H}^2 
+\frac{1}{2}
         \|\widehat{\theta}_h(t)-\theta(t)+\ell(\widehat{\varphi}_h(t)-\varphi(t))\|_H^2 
\\ \notag 
&+ \frac{\ell^2}{2}\int_0^t\|\nabla(\overline{\varphi}_h(s)-\varphi(s))\|_H^2\,ds 
  + \frac{1}{2}\int_0^t\|\nabla(\overline{\theta}_h(s)-\theta(s))\|_H^2\,ds 
\\ \notag 
&\leq C_1 h 
 + 2\ell^2 \|\pi'\|_{L^{\infty}(\mathbb{R})}
       \int_{0}^{t} \|\widehat{\varphi}_h(s)-\varphi(s))\|_{H}^2\,ds 
\\ \notag 
&\,\quad+\left(\frac{3}{2}\ell^2 + \frac{1}{2}\right)\int_0^t
   \|\widehat{\theta}_h(s)-\theta(s)+\ell(\widehat{\varphi}_h(s)-\varphi(s))\|_H^2\,ds 
+\frac{1}{2}\|\overline{f}_h-f\|_{L^2(0, T; H)}^2. 
\end{align*}
Therefore, 
by applying the Gronwall lemma, 
we can obtain \eqref{semierroresti1}.
\end{proof}

\begin{prth2.3}
Since $W^{1, 1}(0, T; H) \subset L^{\infty}(0, T; H)$, 
we have 
\begin{align}\label{fhf1}
\|\overline{f}_h - f\|_{L^2(0, T; H)}^2 
&= \frac{1}{h^2}\sum_{k=1}^{N}
      \int_{(k-1)h}^{kh} \left\|\int_{(k-1)h}^{kh} (f(s)-f(t))\,ds\right\|_H^2\,dt 
\\ \notag 
&\leq \frac{1}{h^2}\sum_{k=1}^{N}
            \int_{(k-1)h}^{kh} \left(\int_{(k-1)h}^{kh} \|f(s)-f(t)\|_{H}\,ds\right)^2\,dt 
\\ \notag 
&\leq \frac{2\|f\|_{L^{\infty}(0, T; H)}}{h^2}\sum_{k=1}^{N}
          \int_{(k-1)h}^{kh} \left(\int_{(k-1)h}^{kh} \|f(s)-f(t)\|_{H}^{1/2}\,ds\right)^2\,dt 
\\ \notag 
&\leq \frac{2\|f\|_{L^{\infty}(0, T; H)}}{h}\sum_{k=1}^{N}
          \int_{(k-1)h}^{kh} \left(\int_{(k-1)h}^{kh} \|f(s)-f(t)\|_{H}\,ds\right)\,dt. 
\end{align}
Here it holds that 
\begin{align}\label{fhf2}
&\frac{2\|f\|_{L^{\infty}(0, T; H)}}{h}\sum_{k=1}^{N}
          \int_{(k-1)h}^{kh} \left(\int_{(k-1)h}^{kh} \|f(s)-f(t)\|_{H}\,ds\right)\,dt 
\\ \notag 
&= \frac{2\|f\|_{L^{\infty}(0, T; H)}}{h}\sum_{k=1}^{N}
      \int_{(k-1)h}^{kh} \left(\int_{(k-1)h}^{kh} 
                           \left\|\int_{t}^{s} \partial_t f(r)\,dr \right\|_H\,ds\right)\,dt 
\\ \notag 
&\leq \frac{2\|f\|_{L^{\infty}(0, T; H)}}{h}\sum_{k=1}^{N}
      \int_{(k-1)h}^{kh} \left(\int_{(k-1)h}^{kh} 
                    \left(\int_{t}^{s} \left\|\partial_t f(r)\right\|_H\,dr\right)\,ds\right)\,dt 
\\ \notag 
&\leq 2\|f\|_{L^{\infty}(0, T; H)}\|\partial_t f\|_{L^1(0, T; H)}h.   
\end{align}
Therefore, combining Lemma \ref{semierroresti}, \eqref{fhf1} and \eqref{fhf2},   
we can prove Theorem \ref{erroresti}.
\qed 
\end{prth2.3}

\begin{remark}\label{about.conv.of.fh}
The above argument can be used also to the check that 
\begin{align}\label{conv.of.fh}
\overline{f}_h \to f \quad\mbox{strongly in}\ L^2(0, T; H) 
\end{align}
as $h\searrow0$\pier{,} 
in the case when $f$ is just in $L^2(0, T; H)$. 
Indeed, by density, for all $\ep > 0$ 
there exists a function $g \in W^{1, 1}(0, T; H)$ 
such that 
$$
\|f - g\|_{L^2(0, T; H)} < \ep.  
$$
By fixing $g$ and introducing $\overline{g}_h$ as in \eqref{overandunderline}, 
with 
$$
g_k := \frac{1}{h}\int_{(k-1)h}^{kh} g(s)\,ds 
$$
for $k = 1, ..., N$, 
the reader can easily verify that 
$$
\|\overline{f}_h - \overline{g}_h\|_{L^2(0, T; H)} < \ep
$$
as well. 
Then there exists $\widehat{h}>0$ sufficiently small such that 
$$
\|g - \overline{g}_h\|_{L^2(0, T; H)} < \ep
$$
for all $h \in (0, \widehat{h})$ (cf.\ \eqref{fhf1} and \eqref{fhf2}). 
Hence this gives a proof of \eqref{conv.of.fh}. 
\end{remark}

\section*{Acknowledgments}
The research of PC is supported by the Italian Ministry of Education, 
University and Research~(MIUR): Dipartimenti di Eccellenza Program (2018--2022) 
-- Dept.~of Mathematics ``F.~Casorati'', University of Pavia. 
In addition, PC gratefully acknowledges some other 
support from the MIUR-PRIN Grant 2015PA5MP7 
``Calculus of Variations'' 
and the GNAMPA (Gruppo Nazionale per l'Analisi Matematica, 
la Probabilit\`a e le loro Applicazioni) of INdAM (Isti\-tuto 
Nazionale di Alta Matematica). 
The research of SK is supported by JSPS Research Fellowships 
for Young Scientists (No.\ 18J21006) 
and JSPS Overseas Challenge Program for Young Researchers. 
%
%
%

\end{document}